\documentclass[letterpaper, 10 pt, conference]{ieeeconf}  

\IEEEoverridecommandlockouts                              

\usepackage{graphics} 
\usepackage{epsfig} 
\usepackage{mathptmx} 
\usepackage{times} 
\usepackage{amsmath} 
\usepackage{amssymb}  
\usepackage{graphicx}
\usepackage{relsize}
\DeclareMathAlphabet{\mathcal}{OMS}{cmsy}{m}{n}
\SetMathAlphabet{\mathcal}{bold}{OMS}{cmsy}{b}{n}

\title{\LARGE \bf
Performance analysis and optimization of power systems with spatially correlated noise
}

\author{Taouba Jouini$^{1}$, Zhiyong Sun$^{2}$ 
\thanks{*This work has received funding from the European Research Council (ERC) under the European Union's Horizon 2020 research  and innovation program (grant agreement No: 834142).}%
\thanks{$^{1}$Taouba Jouini is with the Department of Automatic Control, LTH, Lund University,
        Ole Römers väg 1,  22363 Lund, Sweden. $^{2}$Zhiyong Sun is with Department of Electrical Engineering,  Eindhoven University of Technology, the Netherlands.
        E-mails:
        \tt\small taouba.jouini@control.lth.se, z.sun@tue.nl.}}%

\graphicspath{{./fig/}}
\usepackage[usenames,dvipsnames]{color}
\usepackage{epstopdf}

\usepackage{amsmath,amsthm,amssymb,mathrsfs,dsfont}
\usepackage{amsfonts}
\usepackage{siunitx}
\usepackage[colorinlistoftodos,prependcaption,textsize=tiny]{todonotes}
\usepackage{tikz}
\usetikzlibrary{matrix}

\usepackage{cite}

\newcommand\oprocendsymbol{\hbox{$\blacksquare$}}

\newcommand{\dd}[0]{\mathrm d}
\newcommand\oprocend{\relax\ifmmode\else\unskip\hfill\fi\oprocendsymbol}

\newcommand{\real}[0]{\mathbb R}

\providecommand{\norm}[1]{\lVert#1\rVert}

\newtheorem{theorem}{Theorem}[section]
\newtheorem{lemma}[theorem]{Lemma}

\newtheorem{corollary}[theorem]{Corollary}

\usepackage{verbatim} 

\newcommand{\tb}[0]{\color{blue}}

\usepackage[normalem]{ulem}
\newcommand\rout{\bgroup\markoverwith{\textcolor{red}{/}}\ULon} 

\usepackage[prependcaption,colorinlistoftodos]{todonotes}

\begin{document}

\maketitle
\thispagestyle{empty}
\pagestyle{empty}

\begin{abstract}
 Based on stochastic differential equations (SDEs), we analyse the overall performance of heterogeneous power systems network, subject to spatially distributed and correlated noise with random initial conditions. We determine bounds on the $\mathcal{H}_2$ norm of the heterogeneous system based on a closed-form of the norm of the homogeneous power system. Then, we formulate possible scenarios for performance optimization {and link these to applications for network design and control problems in power systems}. Our results are corroborated by numerical simulations from Kundur's four-machine two-area network after adaption to our setup.
\end{abstract}


\section{INTRODUCTION}

The electrical grid is witnessing  major changes in its planning and operation, mainly driven by economic and environmental concerns \cite{ackermann2017paving}.
A better understanding of today's deployment of renewable resources in power systems will necessarily go through the analysis of the ramifications of the integration of power electronics on grid stability. During the gradual retirement of synchronous machines and its replacement by DC/AC converters in closed loop with efficient controllers, designed to emulate the electro-mechanical interaction inherently present in synchronous machines, mixed power generation seems to be inevitable, and the operation of DC/AC converters is conducted in a first step in the presence of synchronous machines \cite{markovic2019understanding}.

In this context, Stochastic Differential Equations (SDEs) have gained more and more attention in the literature of power systems, motivated by its potential applications in modeling disturbances ubiquitous in real-life power grid~\cite{milano2013systematic}. A systematic and generic approach on how to model power systems as continuous-time SDEs subject to independent Wiener processes was developed in \cite{milano2013systematic}. Gaussian processes have been adopted to model power fluctuations in~\cite{podolsky2013random}. In \cite{guo2019performance}, the performance and stability analysis of low-inertia power grid were considered with both additive and multiplicative noises which model stochastic inertia behavior. Examples included also non-Gaussian disturbance for wind power uncertainty \cite{chen2019optimal}.

To assess the effect of disturbances on the system stability, different approaches have been adopted. We distinguish two main {avenues}, related to convergence properties of the SDE solution. First, {\em strong} convergence, i.e., with respect to a particular stochastic process trajectory, has been studied extensively in the literature that ranges from stability in probability, almost sure (exponential) synchronization~\cite{russo2018synchronization}, transient stability probability, such as the probability of frequency synchronization or voltage collapse \cite{milano2013systematic}. Second, {\em weak} convergence, i.e., with respect to statistical properties of the solution, has also attracted interest in the analysis of power systems performance.  
A variety of recent considerations in the literature deals with weak convergence via quantitative understanding of the disturbance-to-output behavior. One possibility is $\mathcal{H}_\infty$ norm \cite{weiss2004h}, primarily concerned with peaks in the frequency response following an event and the location of the worst case disturbance. Step response notions like, frequency nadir defined as the worst frequency drop, and the rate of change of frequency (RoCoF) as  the maximal slope of frequency change during transients, are also common metrics in the study of the effect of disturbances on frequency stability. 

While Linear Quadratic Regulator (LQR) formulations \cite{markovic2018lqr} remain sensitive to the choice of the time horizon apparent in the objective function \cite{poolla2019placement}, one attractive approach goes by $\mathcal{H}_2$ norm calculation \cite{wu2015input,tegling2015price,poolla2017optimal}. $\mathcal{H}_2$ norm indicates root mean square or average sensitivity of the system performance to disturbances and is derived from a generalized Lyapunov equation \cite{guo2019performance}. For this, $\mathcal{H}_2$ norm approach has been leveraged, at many occasions in the form of an input-to-output measure, e.g. for the total resistive losses incurred in returning a power network of identical generators with resistive and inductive lines, to a synchronous state \cite{tegling2015price}, common local and inter-area oscillations \cite{wu2015input}, and inertia and damping allocation, with specified capacity and budget constraints for DC/AC converters \cite{poolla2017optimal,poolla2019placement}.

Most of the analytical results consider {\em homogeneous} power system  networks, where all the machines or generation units  are identical and  with uncorrelated white noise of { unit} variance. Only recently, {\em heterogeneous} inertia and damping are considered, within a stochastic setting in \cite{paganini2019global} under mild restrictions, which consist in  machine dynamics  proportional to nominal rating and fixed damping to inertia ratio of all the machines. In a second occasion, two differently parameterized behaviors (grid-forming and grid-following) of closed-loop DC/AC converters have been investigated and compared with simulative examples in \cite{poolla2019placement}, for optimal inertia and damping allocation, but not simultaneously.
In \cite{curi2017control}, the stability of a mixed-generation comprising synchronous machines together with DC/AC converter based on model reduction, was considered but in a deterministic setup that does not include stochastic disturbances.

  Compared to the existing literature, the contributions of this paper can be summarized as follows. We first derive an SDE model of heterogeneous power networks, extending previous models to more realistic setups, where parametric heterogeneity and {spatial} noise correlation with random initial conditions, enter the picture. For this, the generation units (which can be thought of as, either synchronous machines or closed-loop DC/AC converters, or both) have non-uniform inertia and damping. The normally distributed and {spatially} correlated noise {models the practical setup, where a generation unit is affected by its neighbor's noise and hence propagated according to graph Laplacian, e.g. cascaded failures, network outage, and voltage collapse \cite{kundur1994power}}. 
Our second main result consists in the derivation of bounds on the $\mathcal{H}_2$ norm of the heterogeneous power system, based on a closed-form of the norm of the homogeneous power system model. {We then demonstrate its utility by formulating and extending important optimization scenarios in design and control of power system networks}. 

{ In fact,} in a mixed-generation framework, where the interplay between DC/AC converters in closed-loop with a machine emulating controller (e.g., droop  \cite{simpson2013synchronization} and matching  control \cite{arghir2018grid}, virtual synchronous machines \cite{bevrani2017virtual}) and synchronous machines is investigated, we pose possible scenarios for performance optimization, with respect to the derived $\mathcal{H}_2$ norm. We formulate an optimal { susceptance} problem to specify the optimal susceptances in a mixed generation (DC/AC converters and or synchronous machines). We additionally  present an optimal { node-edge assignment problem for an optimal pairing of generation units that  improves the system performance in the $\mathcal{H}_2$ norm sense.}  Finally, we extend the inertia and damping allocation problem from \cite{poolla2017optimal} of DC/AC converters, subject to operational and budget constraints. We validate our analytical results for optimal damping and inertia, on  adapted setup from Kundur's four-machine two-area system \cite{kundur1994power}.

The remainder of this paper unfurls as follows. Section~\ref{sec:modeling} formulates and derives the heterogeneous SDE power systems model, starting from the classical swing equation. In Section \ref{sec:gramians}, we determine a closed form of system  $\mathcal{H}_2$ norm, and interpret its dependence on network parameters. In Section \ref{sec:optim}, we formulate  possible optimization scenarios with respect to system norm accounting for  optimal susceptances and network topology, as well as an  extension of inertia and damping allocation for proper deployment of DC/AC converters in the presence of synchronous machines. Finally, Section \ref{sec:sims} validates our results by numerical simulations of an adapted Kundur's 4-machine 2-area system.

\section{Modeling  of power systems with correlated noise}
\label{sec:modeling}
 We consider a heterogeneous power systems model, defined by a graph ${G}=(\mathcal{V},\mathcal{E})$ of an undirected network, where  $\mathcal{V}$ is the set of $n$ heterogeneous generation units (i.e., buses), where inductive load with constant susceptance is considered and absorbed in the lines (e.g. after Kron reduction \cite{dorfler2012kron}). Let $\mathcal{E}$ be the set of $m$ edges (purely inductive lines) with weight susceptance $b_{e}>0, \, e\in\mathcal{E}$.
We denote by $\mathcal{B}\in\real^{n\times m}$ the incidence matrix of the graph ${G}$, and by $\mathcal{N}_i$  the neighbor set of the $i$-th generation unit (DC/AC converter  in closed-loop with droop control or synchronous machines). The voltage magnitude $V_i$ at the $i$-th bus is assumed to be constant and equal to one per unit. Under the approximation of quasi-stationary steady state, the swing equation of the $i$-th generation with inertial constant $m_i>0$, damping coefficient $d_i>0$, and (virtual) voltage phase angle $\theta_i\in\real$ describes the $i$-th node dynamics as follows,
\begin{align}
\label{eq:unit-i}
m_i \ddot\theta_i+ d_i \dot\theta_i = P_{m,i}-P_{e,i}+\eta_i\,,
\end{align}
where $P_{m,i}\in\real$ is constant mechanical power, and $P_{e,i}=\sum_{j\in \mathcal{N}_i} b_{ij} V_i V_j \sin(\theta_i-\theta_j)=\sum_{j\in \mathcal{N}_i} b_{ij}\sin(\theta_i-\theta_j)$ is the electrical power injected from $i$-th generation into the neighbor set $\mathcal{N}_i$ and vice versa. The disturbance $\eta_i(t),\, t>0$ stands for load fluctuations in renewable generation for DC/AC converters, or generator outages for synchronous machines. 

{The graph G is described by} the weighted Laplacian matrix $L=\mathcal{B}\, \Gamma\, \mathcal{B}^\top\in\real^{n\times n}$, in function of $\Gamma=\text{diag}(b_{e})_{e\in\mathcal{E}}\in\real^{m\times m}$, with eigenvalues of a non-decreasing order $\lambda_1(L)=0<\lambda_2(L)\leq\dots \leq\lambda_n(L)$. 
Let $\dd w_i(t)$ be the increment of the $i$-~th {standard Wiener process $w_i(t)$} that results in the disturbance $\eta_i(t)=\frac{\dd w_i(t)}{\dd t}$. {The disturbance $\eta_i(t)$, resulting from the process $w_i(t)$ at node $i$ is correlated to  the disturbance $\eta_j(t)$, resulting from $w_j(t)$, for node $j \in \mathcal{N}_i$ with given covariance matrix $Q= \Sigma \Sigma^\top, \;\Sigma=\gamma^{1/2} L^{1/2}$, where $\gamma>~0$ models the intensity of the noise diffusion. For more general noise diffusion functions, see  \cite{russo2018synchronization}}. 
For identical and uncorrelated noise, the disturbance in {\eqref{eq:unit-i}} corresponds to that adopted in \cite{tegling2015price, poolla2017optimal}.


Let $\omega^*>0$ be synchronous frequency and $\theta^{*}\in\real^n$ be the angles at steady state. After a linearization around a stable (i.e., synchronous) operating point $[\theta^{*\top},\omega^{*}\mathds{1}^\top_n]^\top$, the electrical power can be approximated by $P_{e,i}\approx \sum_{j\in\mathcal{N}_i}b_{ij}(\theta_i-\theta_j)$, and we obtain the small-signal power systems model described by the following linear SDE, 
\begin{align}
\label{eq: lin-swing}
\begin{bmatrix}
	\dd\theta \\  \dd \omega
	\end{bmatrix} &=\begin{bmatrix}
	0 & I \\ -M^{-1}L & -M^{-1} D
	\end{bmatrix} \begin{bmatrix}
	\theta \\ \omega
	\end{bmatrix} \dd t
	+\begin{bmatrix}
	 0\\  M^{-1} { \gamma^{1/2} \; L^{1/2}}
	\end{bmatrix} \dd {W}, \\
y&=\begin{bmatrix}
{L}\, \theta \\ \omega
\end{bmatrix}, \,\begin{bmatrix}
\theta_{0}\\  \omega_{0}
\end{bmatrix}\sim \mathcal{N}(\xi_0,\Sigma_0\Sigma_0^\top),\nonumber 
\end{align}
where we define the angles vector $\theta=\begin{bmatrix}
\theta_1\dots \theta_n
\end{bmatrix}^{\top}\in\real^{n}$, frequency vector $\omega=\begin{bmatrix} \omega_1 \dots  \omega_n\end{bmatrix}^\top\in\real^{n}$, and { standard} Wiener process increments vector  $\dd {W}=\left[\begin{smallmatrix}\dd w_1 \dots \dd w_n \end{smallmatrix}\right]^\top\in\real^{n}$. The vector ${P}_{m}=\begin{bmatrix}
{P}_{m,1} \dots {P}_{m,n}
\end{bmatrix}^\top\in\real^{n}$ represents constant (mechanical) input that can be lumped into the increments $\dd W$ as in \cite{bamieh2013price}. The matrices  $M, D\in\real^{n\times n}$, are positive diagonal matrices whose entries are the non-uniform inertia and damping values, each denoted by $m_i$ and $d_i$ for $i=1\dots n$. The { identity} matrix $I$ { is defined with} appropriate dimensions.

The output $y\in\real^{2n}$  represents phase cohesiveness and frequency drift. The vector $\begin{bmatrix} \theta_0^\top & \omega_0^\top \end{bmatrix}^\top$ lumps the initial states that are normally distributed  random variables with mean vector $\xi_0=\mathds{E}[\begin{bmatrix} \theta^\top_0 & \omega^\top_0 \end{bmatrix}^\top]$, and covariance matrix $Q_0=\Sigma_0 \Sigma_0^\top$, where $\mathds{E}[\cdot]$ denotes the expectation operator. Moreover, we  assume that the initial conditions are independent of the Wiener processes. 

The linearized Swing dynamics in { \eqref{eq: lin-swing}} describe both synchronous machines, and DC/AC converters, through machine emulation controllers, whose prominent feature is to endow DC/AC converters with dynamics of synchronous machines \cite{arghir2018grid}.

\section{Performance analysis}
\label{sec:gramians}
We derive bounds on the system $\mathcal{H}_2$-norm { for the power system model \eqref{eq: lin-swing}} based on insights provided by the norm of the homogeneous power system and in particular, its dependence on key network parameters: inertia, damping {and noise diffusion}.
The $\mathcal{H}_2$ norm of system \eqref{eq: lin-swing} is expressed as a function of the controllability Gramian $P$ by,
\begin{align}
\label{eq:def-H2}
\! \! \! \!\!\!\!\!\!\norm{\mathcal{G}}_2^2=\lim_{t\to\infty}\mathds{E}\{y{(t)}^\top y{ (t)}\}=\text{trace}(C^\top\, P\, C)\, ,
\end{align}
{ where  we denote $y(t)=C\, X(t)$ with $X=\begin{bmatrix}
\theta^\top & \omega^\top \end{bmatrix}^\top$ and  $C=\begin{bmatrix}
L & 0 \\ 0 & I
\end{bmatrix}$.}
This implies that, the system $\mathcal{H}_2$ norm is the trace of the controllability Gramian $P$ weighted by the output matrix $C$ { and satisfying $A P+A^\top P =-R\,R^\top$, where
\[
A=\begin{bmatrix}
    0 & I \\ -M^{-1} L & -M^{-1}D
 \end{bmatrix},\quad R=\begin{bmatrix}
     0 & M^{-1}{\gamma^{1/2} L^{1/2}}\end{bmatrix}^\top.
\]
}

\subsection{Special case: Homogeneous system parameters} 
Consider the continuous-time LTI system $\mathcal{G}_{hom}$ with the state-space representation in {\eqref{eq: lin-swing}} and  homogeneous parameters, that is, the inertia and damping are uniform and described by $M=m\cdot I$, and $D=d\cdot I $, with $m,d>0$.

For this special case, an explicit formula of the system $\mathcal{H}_2$ norm is given { by} the following lemma.

\begin{lemma}
\label{lem: main-result}
Consider the power networks in { \eqref{eq: lin-swing}} with homogeneous inertia and damping values described by $\mathcal{G}_{hom}$. The squared $\mathcal{H}_2$ norm defined in \eqref{eq:def-H2} is given by,  
\begin{align}
\label{eq:H2-norm}
\norm{\mathcal{G}_{hom}(m,d)}_2^2&= \frac{\gamma}{2 \, d} \mathlarger{\sum}_{i=2}^{n}  \left(\lambda^2_i(L) +\frac{ \lambda_i(L)}{ m}\right) 
\end{align}
\end{lemma}
\begin{proof}
 Note that the marginal stability of the system matrix~$A$, (see  \cite{bamieh2013price, poolla2017optimal}) guarantees the existence of a unique positive semi-definite matrix, as solution to the Lyapunov equation 
 $ A\, P+P \, A^\top =- R\, R^\top$, which holds for the system~$\mathcal{G}_{hom}$ with homogeneous inertia and damping matrices.
By spectral decomposition, we write $\mathcal{H}_2$ norm of $\mathcal{G}_{hom}$, as the sum of the norms associated to each individual {mode} after {modal} coordinate transformation.

For this,  we consider the homogeneous system $\mathcal{G}_{hom}$ and the following system: 
  \begin{align}
  \label{eq: transf}
 \dd \Theta&=\left[\begin{smallmatrix} 0 & I \\ -\frac{1}{m} \Lambda & - \frac{d}{m} \cdot I  \end{smallmatrix}\right] \Theta \,\dd t+  \left[\begin{smallmatrix} 0  \\   \frac{1}{m} \gamma^{1/2}\Lambda^{1/2}  \end{smallmatrix} \right]\; \dd {\widetilde W},\\
  \widetilde y&=  \left[\begin{smallmatrix}  \Lambda  & 0 \\ 0 & I \end{smallmatrix}\right] \Theta\, ,
\nonumber  \end{align}
  where we introduce the coordinate transformation $\Theta=\left[\begin{smallmatrix}
  (V^\top  \theta)^{\top}, &   (V^\top  \omega)^{ \top} \end{smallmatrix}\right]^\top$, in which $V$ is an orthogonal matrix whose columns are right eigenvectors of $L$ and $\Lambda$ is a diagonal matrix whose diagonal entries $\lambda_i\geq 0,\quad i=1\dots n$, are  eigenvalues of $L$.
The transformed system \eqref{eq: transf}  has the same squared $\mathcal{H}_2$ norm of $\mathcal{G}_{hom}$ (see \cite{bamieh2013price}), with $ \dd \widetilde{ {W}}=V^\top\,  \dd {W}$ and $\widetilde{y}=\left[\begin{smallmatrix}
    V^\top & 0 \\ 0 & V^\top 
    \end{smallmatrix}\right]y$.  As a consequence, we obtain $n$- decoupled subsystems $A^i= \left[\begin{smallmatrix} 0 & 1 \\ -\frac{\lambda_i}{m} & -\frac{d}{m} \end{smallmatrix}\right]$ of second order.  We calculate the controllability Gramian $P^i $ for the $i-$th system, associated with ${r}^i=\begin{bmatrix}
    0 & \frac{1}{m} \gamma^{1/2} \lambda^{1/2}_i(L)
    \end{bmatrix}^\top$, which verifies $A^i \,  P^i  + P^i  \, A^{i\top}=-r^i \,  r^{i\top}$.  By solving the Lyapunov equation for  $ P^i =\left[\begin{smallmatrix}
p^i_1 & p^i_2 \\ p^i_2  & p^i_3
\end{smallmatrix}\right]$, we arrive at $p^i_1=\frac{\gamma}{2 \, d},\, p^i_2=0,\, p^i_3=\frac{\gamma\,  \lambda_i}{2\, d \, m}$.
It follows that $\text{trace}(( C^i)^\top P^i  C^i)=\lambda^2_i(L)\, p^i_1+p^i_3= \lambda_i^2(L) \frac{\gamma}{2 \, d}+\frac{\gamma\,  \lambda_i(L)}{2\, d \, m}$, with $C^i=\left[\begin{smallmatrix}
\lambda_i(L) & 0 \\ 0 & 1
\end{smallmatrix}\right]$. Since the mode $\lambda_1(L)=0$ is uncontrollable (by $k^\top_0 P=0,\;k^\top_0=\begin{bmatrix}  \mathds{1}^\top_n D &  \mathds{1}^\top_n M \end{bmatrix})$, and hence does not contribute to the system $\mathcal{H}_2$ norm, we find \eqref{eq:H2-norm}.
\end{proof}

\subsection{Interpretation and implications for heterogeneous case}
\label{subsec: interp}
The $\mathcal{H}_2$ norm in \eqref{eq:def-H2} is primarily concerned with the {\em overall} performance of system { \eqref{eq: lin-swing}}, and regarded as the energy amplification for the input at each generation, being a unit impulse.

For homogeneous setup, and as a direct consequence of Lemma \ref{lem: main-result}, the $\mathcal{H}_2$ norm in \eqref{eq:H2-norm} increases with noise { diffusion} parameter  $\gamma>0$, and decreases with damping $d>0$ and inertia $m>0$. 

In fact, the trace of the Gramian \eqref{eq:H2-norm} is inversely related to the average energy or average controllability in all directions in the state space. Note that, by rewriting \eqref{eq:H2-norm} as $\norm{\mathcal{G}_{hom}(m,d)}_2^2=\frac{1}{2} \sum_{i=2}^n  f_i(m,d)$, $f_i(m,d)= \frac{\gamma}{d}  \lambda_i(L) \left(\lambda_i(L) +\frac{1}{ m}\right)$, defines the average controllability centrality for the nodes. The nodes with least centrality minimize the $\mathcal{H}_2$ norm \cite{summers2015submodularity}.

By defining upper and lower bound for the inertia $\overline M=\max\{m_1,\dots ,m_n\},\,  \underline M=\min\{m_1,\dots ,m_n\}$ and damping $\overline D=\max\{d_1,\dots ,d_n\},\, \underline D=\min\{d_1,\dots ,d_n\}$, we can find an upper bound and a lower bound for the 
$\mathcal{H}_2$ norm of the heterogeneous power systems, and given by
\begin{align}
\label{eq:bound}
\!\norm{\mathcal{G}_{hom}(\overline M,\overline D) }_2^2\leq \norm{\mathcal{G}}_2^2 \leq \norm{\mathcal{G}_{hom}(\underline M,\underline D)}_2^2\;.
\end{align}
Note  that in general, it is not always possible to calculate the $\mathcal{H}_2$ norm in \eqref{eq:def-H2}, and one can use the upper bound provided in \eqref{eq:bound} to account for worst-case system performance, after a disturbance, while satisfying specific operation constraints. { This is demonstrated in the next section}.

\section{Performance optimization}
\label{sec:optim}
Motivated by the recently examined heterogeneous power system  models \cite{markovic2019understanding}  that consider a mixed-generation model, partitioned into DC/AC converters in closed-loop with a controller (e.g. droop control  \cite{simpson2013synchronization}) and synchronous machines (with eventually a governor control), we consider optimization problems {that} minimize the system $\mathcal{H}_2$ norm in \eqref{eq:def-H2} for the proper deployment of { mixed generation units} (DC/AC converters in the presence of synchronous machines).



%

{
\subsection{Scenario 1: Susceptance   optimization problem} 
For a given graph topology and in particular, a fixed node-incidence matrix $\mathcal{B}\in\real^{n\times m}$, we aim to determine the optimal susceptance matrix $\Gamma\in\real_+^{m\times m}$, and in particular the optimal allocation of the susceptance values $\Gamma_{ii}=b_{e}> 0,\, e\in\mathcal{E}$, along the edges to  optimize the system performance, along with power flow at steady state $P^*\in\real^{n}$ at all the generation units.  

For this, we utilize the upper bound on the $\mathcal{H}_2$ norm $\norm{\mathcal{G}}_2^2$ found in~\eqref{eq:bound}. Let $\theta^*_i\in\real$ denote well-known angles of the $i$-th generation unit at steady state. Then the optimization problem is formulated as, 
\begin{align}
\label{eq: op1}
 \min_{\Gamma,\,  P^*}  \quad  \frac{\gamma\; \underline D^{-1}}{2} \, & \left(\vert\vert \lambda(\mathcal{B}\, \Gamma\, \mathcal{B}^\top)\vert\vert_2^2 +{\underline M}^{-1}{\vert\vert\lambda(\mathcal{B}\, \Gamma\, \mathcal{B}^\top)\vert\vert_1}\right)\\ \nonumber
\text{subject to} & \quad  P^*=\mathcal{B}\,\Gamma\, \mathcal{B}^\top \theta^*, \quad  \text{[power balance]} \\ \nonumber
 & \quad  \Gamma_{ii}>0,\quad \Gamma_{ij}=0, \quad i,j\in\mathcal{V} \\ \nonumber
 & \quad \underline{b}_{e}  \leq \Gamma_{ii} \leq \overline{b}_{e} 
  \quad  \text{[capacity constraints]}
  \\ \nonumber
  & \quad \sum_{e\in\mathcal{E}} c_{e}(b_{e}) = K \quad  \text{[cost constraints]}
\end{align}
where $\vert\vert \cdot \vert\vert_2$ and $\vert\vert \cdot \vert\vert_1$ denote, respectively, the Euclidean $\ell_2$ norm and $\ell_1$ vector norm. The parameters  $\underline{b}_e$ and $\overline{b}_e$ are the minimal and  maximal values for the edge susceptances that representing operational capacity constraints.
Notice that, the power flow balance equality indicates that at steady state, the overall power input consisting of converter DC power and total power of synchronous machines sums up to zero. The cost $c_{e,i}(b_{e,i})$ is an increasing function of $b_{e,i}$ that accommodates operational cost of installing the susceptance $b_e$ at the $i$-th edge $e\in\mathcal{E}$, with $K>0$ being the total monetary budget.




\subsection{Scenario 2: Node-edge assignment problem} 
\label{sec:optimal_H2}
{Consider a mixed generation setup with a fixed number of DC/AC converters and machines. For a given number of transmission lines $m\in\mathbb{N}$, known susceptances $b_e, \; e\in~\mathcal{E}$, and angle values at steady-state denoted by $\theta^*\in\real^n$, we search for the optimal pairing of generation units $(i,j)=e,\; e\in\mathcal{E}$ with $i,j\in\mathcal{V}$ encoded in $\mathcal{B}\in\real^{n\times m}$ and power flow at steady state $P^*\in\real^{n}$, at all the generation units, that minimizees the upper bound on $\mathcal{H}_2$ norm in \eqref{eq:bound}. This can be formulated as follows,
\begin{align}
\label{eq: op2}
 \min_{\mathcal{B},\, P^*} \quad  \frac{\gamma\; \underline D^{-1}}{2} &\left(\vert\vert \lambda(\mathcal{B}\, \Gamma\, \mathcal{B}^\top)\vert\vert_2^2 +{\underline M}^{-1}{\vert\vert\lambda(\mathcal{B}\, \Gamma\, \mathcal{B}^\top)\vert\vert_1}\right)\\ \nonumber
\text{subject to} & \quad P^*= \mathcal{B} \Gamma \mathcal{B}^\top \; \theta^* \quad  \text{[power balance]}& \\ \nonumber
& \mathcal{B}=\left\{ \begin{array}{ll}
                  \mathcal{B}_{i,e}=1,\\ 
                  \mathcal{B}_{j,e}=-1,\; e=(i,j),\; i,j\in\mathcal{V}\\
                \mathcal{B}_{k,e}=0,\; k\notin \{i,j\}\\
                \end{array}
              \right.\; 
\end{align}
}
The optimization problem in \eqref{eq: op2} relies on  discrete combinatorics to choose a pair of nodes $(i,j)$ and relates to classical and well-known optimization scenarios in network topology design, see e.g., \cite{de2016growing}.



%
 
\subsection{Discussion and other optimization scenarios} 
Even though Scenarios 1 and 2 tackle the performance optimization  from two different angles, we can combine both formulations to obtain the following more general power network design problem using a min-max formulation,
\begin{align}
\label{eq: op3}
 \min_{\mathcal{B},\;\Gamma,\, P^*} \; \!\! \frac{n\,\gamma\, \underline D^{-1}}{2}\!\!&\left(\vert \vert\lambda(\mathcal{B}\, \Gamma \, \mathcal{B}^\top)\vert\vert^2 _{\infty}+{\underline M}^{-1}\vert\vert\lambda(\mathcal{B} \,\Gamma\, \mathcal{B}^\top)\vert\vert_{\infty}\right)\\ \nonumber
\text{subject to}  \quad&  P^*=\mathcal{B}\,\Gamma\, \mathcal{B}^\top \theta^*, \\ \nonumber
&\Gamma_{ii}>0,\quad \Gamma_{ij}=0,\quad i,j\in\mathcal{V} \\ \nonumber
& \underline{b}_{e}  \leq \Gamma_{ii} \leq \overline{b}_{e}, \; \quad \sum_{e\in\mathcal{E}} c_{e}(b_{e}) = K \\ 
&\mathcal{B}=\left\{ \begin{array}{ll}
                  \mathcal{B}_{i,e}=1,\\ 
                  \mathcal{B}_{j,e}=-1,\; e=(i,j),\; i,j\in\mathcal{V}\\ \nonumber
                \mathcal{B}_{k,e}=0,\; k\notin \{i,j\}\\ \nonumber
                \end{array}
              \right.\; \\ \nonumber
&\mathcal{B}\in\real^{n\times m}, \, \Gamma\in\real_+^{m\times m}, \, P^*\in\real^n
\end{align}
where $\vert\vert\cdot \vert\vert_\infty$ is the maximum vector norm. The cost function in \eqref{eq: op3} accounts for the worst-case eigenvalue $\lambda_{max}$, whereas the cost function in \eqref{eq: op1} and \eqref{eq: op2} accounts for the sum over all the eigenvalues of the Laplacian.

The difference between the cost functions in \eqref{eq: op1}, \eqref{eq: op2} and in \eqref{eq: op3} can be derived from $\norm{\lambda}_\infty \leq\norm{\lambda}_2\leq \norm{\lambda}_1\leq n \norm{\lambda}_\infty$. 
This shows that with  a smaller number of generation units~$n$, we can get a better estimate of \eqref{eq: op1} and \eqref{eq: op2}, using the cost function \eqref{eq: op3}. 




From the bounds in \eqref{eq:bound}, and $\vert \vert \mathcal{G}\vert\vert_2^2 - \vert \vert \mathcal{G}_{hom}(\underline M,\underline D)\vert\vert_2^2\leq \vert \vert \mathcal{G}_{hom}(\underline M,\underline D)\vert\vert_2^2 - \vert \vert \mathcal{G}_{hom}(\overline M,\overline D)\vert\vert_2^2 $, we estimate the gap between the value function $\vert\vert \mathcal{G}_{hom}(\underline D, \underline M)\vert\vert_2^2$ in \eqref{eq: op1} and that of $\mathcal{H}_2$ norm of the heterogeneous system $\vert\vert \mathcal{G}\vert\vert_2^2$ in \eqref{eq:def-H2}, as follows,
\begin{align*}
   \!\!\!\vert \vert \mathcal{G}\vert\vert_2^2 -\!\!\! \vert \vert \mathcal{G}_{hom}(\underline M,\underline D)\vert\vert_2^2 \leq \!\!\!
    \frac{\gamma}{2\;\overline D \cdot\underline D}\!\!\left(\delta _D \norm{\lambda(L)}_2^2+\!\!\!\frac{\delta_{MD}}{\underline{M}\cdot\overline{M}}\norm{\lambda(L)}_1 \!\!\right)&
\end{align*}
where $\delta_D=\overline D-\underline D$ and $\delta_{MD}=\overline D\;\overline M-\underline M \;\underline D$.

We make the following remarks: The further apart the inertia $m_i$, and the damping $d_i$ with $i=1\dots n$, of the individual generation units (synchronous machines or DC/AC converters with droop control) are spread, the wider the difference between the $\mathcal{H}_2$ norm of the heterogeneous and homogeneous system will get. 
Note also that less connectivity of the network (in the sense of the smallest positive eigenvalue $\lambda_2>0$, also termed Fiedler eigenvalue), implies smaller 1- and 2-norm of the eigenvalue vector $\lambda$, and smaller difference between the two $\mathcal{H}_2$ norms. One can deduce that {\em sparsity promotes homogeneity}: in a sparse power system network, $\mathcal{H}_2$ norm of parameter homogeneous system is a good approximation of the system performance.
Finally, we note also that the gap between the two norms decreases with smaller diffusion parameter $\gamma$. 

While we restrict our attention to small-signal models (linearized models in \eqref{eq: lin-swing}), it is noteworthy that network connectivity encoded in the node-incidence $\mathcal{B}$ and line susceptance matrix $\Gamma$, plays a determinant role in achieving synchronization in non-linear power system models. In this case, a trade-off must be taken into consideration in the design of the matrices $\Gamma$ and $\mathcal{B}$, see e.g. \cite{russo2018synchronization}.
}

{ Finally, we extend the optimal inertia allocation problem in \cite{poolla2017optimal}, with addition to the allocation of the damping coefficients while respecting power sharing}. { For this}, assume that the total amount of power associated with synchronous machines at steady state, given by $\sum_{i=1}^{n_G} P^*_{{ G},i}$ is negative, { where $P_{{ G},i}^*$ is the steady state power at the $i$-th machine and $n_G$ is the total number of generators}. This guarantees that the power balance constraint in the next optimization problem is feasible, and can be satisfied, by including resistive load models, e.g., absorbed in the lines as in \cite{tegling2015price}, and not only inductive (in which case, the power flowing from the generation into the load is negative, by common convention, see  \cite{simpson2013synchronization}).
For fixed values of inertia and damping coefficients of synchronous machines, a total monetary budget, operating capacity constraints, and prescribed power sharing ratios, we aim in the remainder to determine the optimal distribution of inertia ${m}_{C,i}$ and damping ${d}_{C,i}$ values, among $n_C$ DC/AC converters, that would minimize~\eqref{eq:def-H2}. This amounts to the following optimization problem:
\begin{align}
\label{eq: op4}
\min_{m_{C,i}, d_{C,i} }\quad &   \text{trace}(C^\top P\, C) \\
\text{subject to}&  \quad  \underline{M}_i \leq m_{C,i} \leq \overline M_i, \qquad  \text{[capacity  constraints]} \nonumber \\  \nonumber
&\quad  \underline{D}_i \leq d_{C,i} \leq \overline D_i \nonumber\\ \nonumber 
& \quad \sum_{i=1}^{n_C} m_{C,i}= K \qquad \quad  \text{[budget  constraints]} \\\nonumber
& \quad \sum_{i=1}^{n_C} P^*_{{ C},i}=\overline P \qquad \quad \text{[power balance]}  \\
& \quad  \frac{\vert P^*_{{ C},i}\vert }{d_{C,i}}=\frac{\vert P^*_{{ C},j}\vert }{d_{C,j}}  \qquad   \text{[power sharing]}\nonumber
\end{align}
 where $\overline P:=-\sum_{i=1}^{n_C} P^*_{{ G},i}$,  $\overline M_i, \, \underline{M}_i>0$ and $\overline D_i, \, \underline{D}_i>0$ correspond respectively to the individual maximal and minimal inertia and damping, for the DC/AC converter at the $i$-th station, $K$ represents budget constraints, and  $r={\vert P^*_{{G},i}\vert}/{d_{G,i}}={\vert P^*_{{ G},j}\vert}/{d_{G,j}}$, for all $i,j=1,\cdots ,n_G$, prescribes power sharing ratio. 
 



\section{Numerical simulations}
\label{sec:sims}
\begin{figure}[h!]
    \centering
    \includegraphics[trim=0cm 5cm 0cm 3.5cm, clip=true ,  scale=0.21]{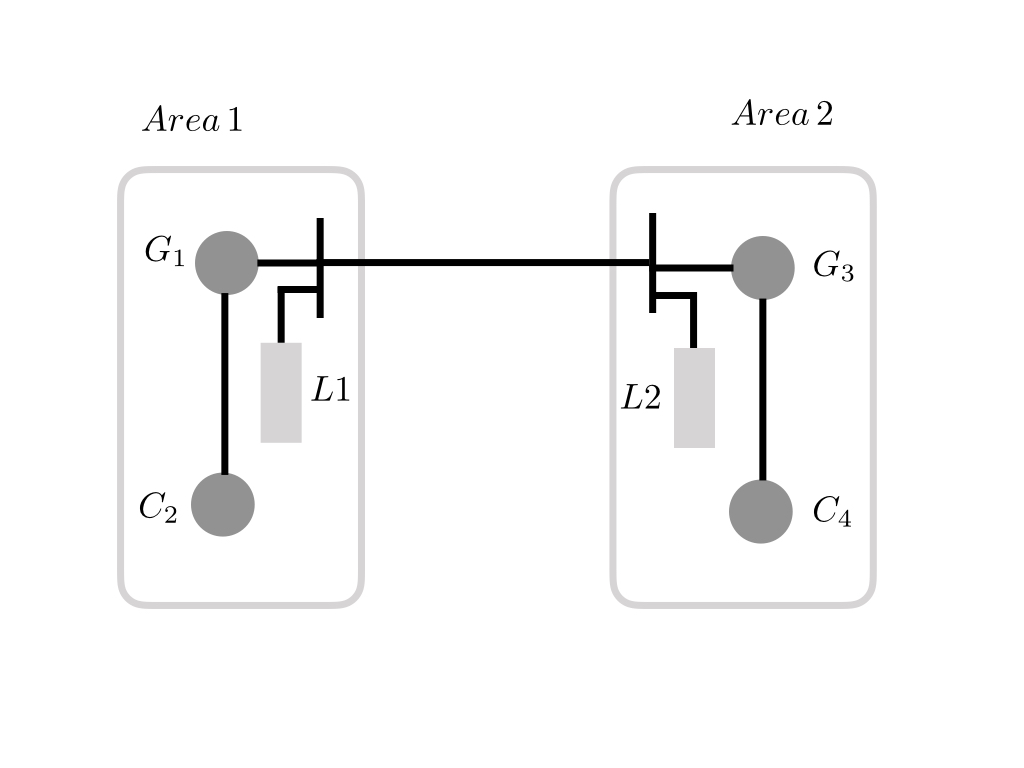}
    \caption{Kundur's four-machine two-area power system composed of two areas, each comprising two machines attached to a load with constant impedance. {The generators $G_2$ and $G_4$ are replaced by DC/AC converters $C_2$ and $C_4$, respectively (in closed-loop with matching control).} All machines and transmission line parameters can be found in~\cite{kundur1994power}.}
    \label{fig:my_label}
\end{figure}
We adopt the linearized MATLAB model of Kundur's four-machine two-area system, depicted in Figure \ref{fig:my_label}, with parameters (in p.u.) from \cite{kundur1994power} and adapt it to our setup. In particular, we replace $G_2$ by DC/AC converter $C_1$ and $G_4$ by DC/AC converter $C_2$, both in closed-loop with matching control { (known to have droop control properties \cite{arghir2018grid})}, index each machine by its area number (1 or 2), set the lines to be purely inductive, integrate continuous correlated noise with intensity $\gamma=0.05$.
The initial conditions  are uncorrelated and normally distributed with mean $\xi_0=\begin{bmatrix}93.077,69.3918,56.5361,45.6552\end{bmatrix}^\top $ and variance $\Sigma_0=\left[\begin{smallmatrix} \sqrt{0.07} \, I_2 & 0 \\ 0 &  \sqrt{0.01}\,  I_2 \end{smallmatrix}\right]$. 
{ The assumption on negative total machines power is satisfied with $\overline P=-P^*_{G1}-P^*_{G2}=0.7778+0.798889=1.5767$, due to the presence of the load $L_1$ and $L_2$. This corresponds to a nominal operation for the synchronous machines $G_1$ and $G_3$ as given by \cite{kundur1994power}.}

We search for the optimal inertia $m_{C,i}$ (in [MW $s^2$/rad]) and the optimal damping $d_{C,i}$ (in [MW s/rad]), minimizing \eqref{eq: op4}, with $n=4$ and $n_C=2$. The upper bounds on inertia and damping are given by $\overline{M_i}=\frac{\vert P_{max,i}\vert}{\max_{t\geq 0}\vert\dot \omega_i(t)\vert}$ and $\overline{D_i}=\frac{\vert P_{max,i}\vert}{\max_{t\geq 0}\vert\omega_i(t)\vert}$, as in \cite{poolla2019placement}, where we denote by $P_{max,i}$ the maximal power of the $i$-th converter. The minimal values are chosen, so that $\underline{M}_1=\underline{D}_1=10$ and $\underline{M}_2=\underline{D}_2=5$, where $m_{C,1}+m_{C,2}=120,\, d_{C,1}+d_{C,2}=40$ { (in SI)}. { By solving} \eqref{eq: op4} using the algorithm from \cite{poolla2017optimal} and MATLAB function $\tt{fmincon}$, { we}  arrive at the optimal inertia and damping values: $m^*_{C,1}=50.00139,\, m^*_{C,2}=69.9987$ and $d_{C,1}^*= 5.0001$ and $d_{C,2}^*=34.9999$. This amounts to { the} $\mathcal{H}_2${ -norm} value of $0.1630$. If we instead allocate the damping and inertia uniformly according to $ \hat m_{C,1}=\hat m_{C,2}=60$ and $ \hat d_{C,1}=\hat d_{C,2}=20$, then the $\mathcal{H}_2$ norm is $0.3217$, which agrees with our predictions from Section~\ref{sec:optimal_H2}. 

{ Figures \ref{fig: area1} and \ref{fig: area2} show the frequency response in simulations at each of the generation units, under spatially correlated noise as in \eqref{eq: lin-swing}. Frequency transients infer power system losses incurred by the generation units to return to synchrony. 
A synchronization at the individual areas (1 and 2) is observed at all the plots, followed by a synchronization at all generation units at the steady state frequency $\omega^*=1$ p.u. 
By comparing the subplots in Figures \ref{fig: area1} and \ref{fig: area2}, representing the frequencies at the converters $\omega_{C,1}$ and $\omega_{C,2}$ (plotted against the frequency of the machines $\omega_{G,1}$ in Area 1 and $\omega_{G,2}$ in Area 2), the optimal allocation of inertia and damping $(m^*_{C,1}, d^*_{C,1})$ at the converters $C_1$ and $(m^*_{C,2}, d^*_{C,2})$ at the converters $C_2$, resulting from solving \eqref{eq: op4}, allows for significantly better transients for both converters, that uniform damping and inertia gains $(\hat m_{C,1}, \hat d_{C,1})=(\hat m_{C,2}, \hat d_{C,2})$ does not achieve.}  
\begin{figure}[h!]
    \centering
    \includegraphics[scale=0.41, trim={0 10cm 0 10cm},clip]{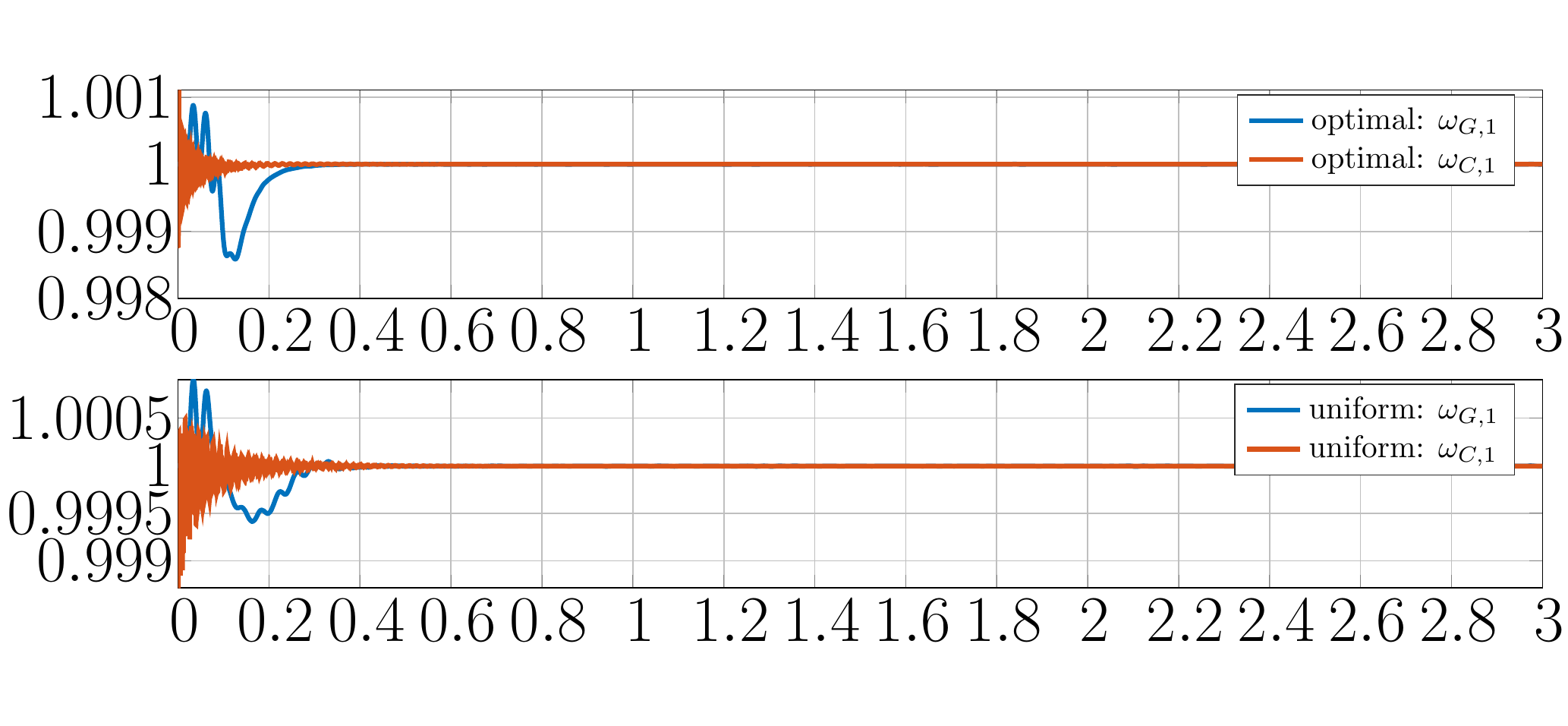}
    \caption{ Frequency responses over time of the synchronous machine $G_1$ and DC/AC converter $C_2$ (in p.u.) in area 1 are plotted in red and blue, respectively, corresponding to two different ways of choosing the inertia and damping for the DC/AC converters. The optimal droop control parameters $(m^*_{C,i},d^*_{C,i}), i=1,2$ solve the optimization problem \eqref{eq: op4} for DC/AC converter, whereas the uniform inertia and damping imply that $(\hat m_{C,1}, \hat d_{C,1})=(\hat m_{C,2}, \hat d_{C,2})$. Optimally allocated inertia $m^*_{C,1}$ and damping $d^*_{C,1}$ improve the system performance by resulting into better transients, and hence less power system losses in area 1. Synchronous machine's inertia and damping are kept fixed.}
   \label{fig: area1}
\end{figure}

\begin{figure}[h!]
    \centering
    \includegraphics[scale=0.41, trim={0 10cm 0 10cm},clip]{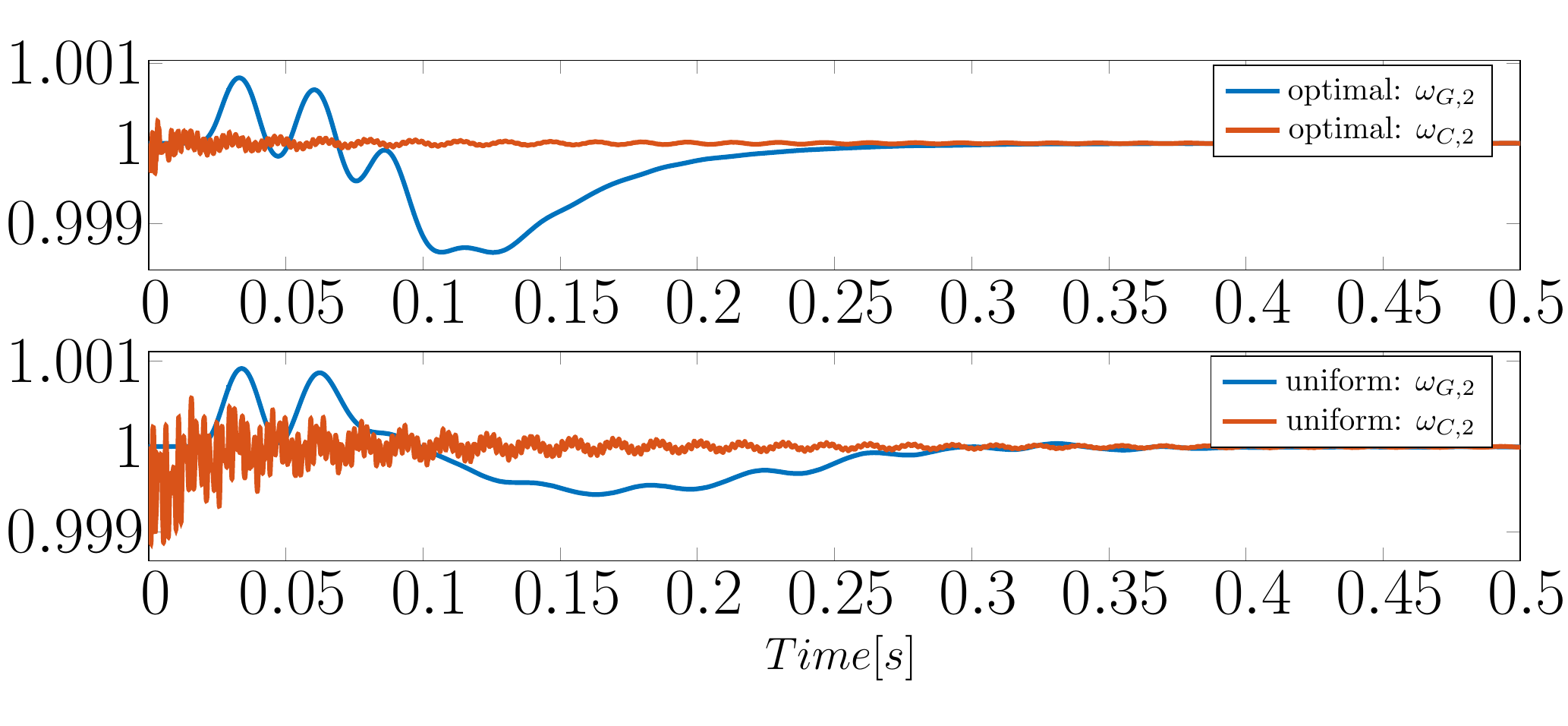}
    \caption{ Frequency responses over time of synchronous machine $G_3$ and DC/AC converter $C_4$ (in p.u.) in area 2 are plotted in red and blue, respectively, with two subplots: First, optimally allocated inertia $m^*_{C,2}$ and damping $d^*_{C,2}$ which improve the system performance by resulting into less transient magnitudes, and hence less power system losses in area 2;  Second, uniform inertia and damping, with $(\hat m_{C,1}, \hat d_{C,1})=(\hat m_{C,2}, \hat d_{C,2})$. Synchronous machine's inertia and damping are kept fixed.
    }
    \label{fig: area2}
\end{figure}

\section{Conclusions}
 Starting from an SDE model for heterogeneous power networks with non-uniform inertia and damping and subject to correlated noise with random initial conditions, we examined the overall network performance { by finding bounds on} $\mathcal{H}_2$ norm. Then, for the mixed-power generation setup consisting of DC/AC converters in closed-loop with { droop control} and synchronous machines, we formulated different scenarios for performance optimization under the derived { bounds on} $\mathcal{H}_2$ norm. Our simulations showcase our findings for the optimal inertia and damping allocation. Future directions include the investigation of network performance for more detailed models, and the study of (approximation) solutions to the proposed optimization schemes.





\section*{ACKNOWLEDGMENT}
The authors would like to thank Dr. Emma Tegling for the insightful comments and discussions.

\bibliographystyle{IEEEtran}
\bibliography{root}

\begin{thebibliography}{10}
\providecommand{\url}[1]{#1}
\csname url@rmstyle\endcsname
\providecommand{\newblock}{\relax}
\providecommand{\bibinfo}[2]{#2}
\providecommand\BIBentrySTDinterwordspacing{\spaceskip=0pt\relax}
\providecommand\BIBentryALTinterwordstretchfactor{4}
\providecommand\BIBentryALTinterwordspacing{\spaceskip=\fontdimen2\font plus
\BIBentryALTinterwordstretchfactor\fontdimen3\font minus
  \fontdimen4\font\relax}
\providecommand\BIBforeignlanguage[2]{{%
\expandafter\ifx\csname l@#1\endcsname\relax
\typeout{** WARNING: IEEEtran.bst: No hyphenation pattern has been}%
\typeout{** loaded for the language `#1'. Using the pattern for}%
\typeout{** the default language instead.}%
\else
\language=\csname l@#1\endcsname
\fi
#2}}

\bibitem{ackermann2017paving}
T.~Ackermann, T.~Prevost, V.~Vittal, A.~J. Roscoe, J.~Matevosyan, and
  N.~Miller, ``Paving the way: A future without inertia is closer than you
  think,'' \emph{IEEE Power and Energy Magazine}, vol.~15, no.~6, pp. 61--69,
  2017.

\bibitem{markovic2019understanding}
U.~Markovic, O.~Stanojev, E.~Vrettos, P.~Aristidou, and G.~Hug, ``Understanding
  stability of low-inertia systems,'' 2019.

\bibitem{milano2013systematic}
F.~Milano and R.~Z{\'a}rate-Mi{\~n}ano, ``A systematic method to model power
  systems as stochastic differential algebraic equations,'' \emph{IEEE
  Transactions on Power Systems}, vol.~28, no.~4, pp. 4537--4544, 2013.

\bibitem{podolsky2013random}
D.~Podolsky and K.~Turitsyn, ``Random load fluctuations and collapse
  probability of a power system operating near codimension 1 saddle-node
  bifurcation,'' in \emph{2013 IEEE Power \& Energy Society General
  Meeting}.\hskip 1em plus 0.5em minus 0.4em\relax IEEE, 2013, pp. 1--5.

\bibitem{guo2019performance}
Y.~Guo and T.~H. Summers, ``A performance and stability analysis of low-inertia
  power grids with stochastic system inertia,'' \emph{arXiv preprint
  arXiv:1903.00635}, 2019.

\bibitem{chen2019optimal}
X.~Chen, J.~Lin, F.~Liu, and Y.~Song, ``Optimal control of {AGC} systems
  considering non-gaussian wind power uncertainty,'' \emph{IEEE Transactions on
  Power Systems}, 2019.

\bibitem{russo2018synchronization}
G.~Russo, F.~Wirth, and R.~Shorten, ``On synchronization in continuous-time
  networks of nonlinear nodes with state-dependent and degenerate noise
  diffusion,'' \emph{IEEE Transactions on Automatic Control}, vol.~64, no.~1,
  pp. 389--395, 2018.

\bibitem{weiss2004h}
G.~Weiss, Q.-C. Zhong, T.~C. Green, and J.~Liang, ``H/sup/spl infin//repetitive
  control of {DC-AC} converters in microgrids,'' \emph{IEEE Transactions on
  Power Electronics}, vol.~19, no.~1, pp. 219--230, 2004.

\bibitem{markovic2018lqr}
U.~Markovic, Z.~Chu, P.~Aristidou, and G.~Hug, ``{LQR}-based adaptive virtual
  synchronous machine for power systems with high inverter penetration,''
  \emph{IEEE Transactions on Sustainable Energy}, vol.~10, no.~3, pp.
  1501--1512, 2018.

\bibitem{poolla2019placement}
B.~K. Poolla, D.~Gro{\ss}, and F.~D{\"o}rfler, ``Placement and implementation
  of grid-forming and grid-following virtual inertia and fast frequency
  response,'' \emph{IEEE Transactions on Power Systems}, vol.~34, no.~4, pp.
  3035--3046, 2019.

\bibitem{wu2015input}
X.~Wu, F.~D{\"o}rfler, and M.~R. Jovanovi{\'c}, ``Input-output analysis and
  decentralized optimal control of inter-area oscillations in power systems,''
  \emph{IEEE Transactions on Power Systems}, vol.~31, no.~3, pp. 2434--2444,
  2015.

\bibitem{tegling2015price}
E.~Tegling, B.~Bamieh, and D.~F. Gayme, ``The price of synchrony: Evaluating
  the resistive losses in synchronizing power networks,'' \emph{IEEE
  Transactions on Control of Network Systems}, vol.~2, no.~3, pp. 254--266,
  2015.

\bibitem{poolla2017optimal}
B.~K. Poolla, S.~Bolognani, and F.~D{\"o}rfler, ``Optimal placement of virtual
  inertia in power grids,'' \emph{IEEE Transactions on Automatic Control},
  vol.~62, no.~12, pp. 6209--6220, 2017.

\bibitem{paganini2019global}
F.~Paganini and E.~Mallada, ``Global analysis of synchronization performance
  for power systems: bridging the theory-practice gap,'' \emph{IEEE
  Transactions on Automatic Control}, 2019.

\bibitem{curi2017control}
S.~Curi, D.~Gro{\ss}, and F.~D{\"o}rfler, ``Control of low-inertia power grids:
  A model reduction approach,'' in \emph{2017 IEEE 56th Annual Conference on
  Decision and Control (CDC)}.\hskip 1em plus 0.5em minus 0.4em\relax IEEE,
  2017, pp. 5708--5713.

\bibitem{kundur1994power}
P.~Kundur, N.~J. Balu, and M.~G. Lauby, \emph{Power system stability and
  control}.\hskip 1em plus 0.5em minus 0.4em\relax McGraw-hill New York, 1994,
  vol.~7.

\bibitem{simpson2013synchronization}
J.~W. Simpson-Porco, F.~D{\"o}rfler, and F.~Bullo, ``Synchronization and power
  sharing for droop-controlled inverters in islanded microgrids,''
  \emph{Automatica}, vol.~49, no.~9, pp. 2603--2611, 2013.

\bibitem{arghir2018grid}
C.~Arghir, T.~Jouini, and F.~D{\"o}rfler, ``Grid-forming control for power
  converters based on matching of synchronous machines,'' \emph{Automatica},
  vol.~95, pp. 273--282, 2018.

\bibitem{bevrani2017virtual}
H.~Bevrani and J.~Raisch, ``On virtual inertia application in power grid
  frequency control,'' \emph{Energy Procedia}, vol. 141, pp. 681--688, 2017.

\bibitem{dorfler2012kron}
F.~Dorfler and F.~Bullo, ``Kron reduction of graphs with applications to
  electrical networks,'' \emph{IEEE Transactions on Circuits and Systems I:
  Regular Papers}, vol.~60, no.~1, pp. 150--163, 2012.

\bibitem{bamieh2013price}
B.~Bamieh and D.~F. Gayme, ``The price of synchrony: Resistive losses due to
  phase synchronization in power networks,'' in \emph{2013 American Control
  Conference}.\hskip 1em plus 0.5em minus 0.4em\relax IEEE, 2013, pp.
  5815--5820.

\bibitem{summers2015submodularity}
T.~H. Summers, F.~L. Cortesi, and J.~Lygeros, ``On submodularity and
  controllability in complex dynamical networks,'' \emph{IEEE Transactions on
  Control of Network Systems}, vol.~3, no.~1, pp. 91--101, 2015.

\bibitem{de2016growing}
M.~H. de~Badyn and M.~Mesbahi, ``Growing controllable networks via whiskering
  and submodular optimization,'' in \emph{2016 IEEE 55th Conference on Decision
  and Control (CDC)}.\hskip 1em plus 0.5em minus 0.4em\relax IEEE, 2016, pp.
  867--872.

\end{thebibliography}

\end{document}